\documentclass[a4paper,11pt]{amsart}
\usepackage{amssymb}
\usepackage{amsmath}
\usepackage{cancel}
\usepackage{mathrsfs}
\usepackage{color}
\usepackage{caption}
\usepackage{subcaption}
\usepackage{hyperref}
\usepackage{backref}
\pdfoutput=1
\usepackage[a4paper,margin=3cm]{geometry}
\DeclareSymbolFont{script}{U}{eus}{m}{n}
\DeclareSymbolFontAlphabet{\amathscr}{script}
\DeclareMathSymbol{\Wedge}{0}{script}{"5E}
\DeclareMathAlphabet{\mathrmsl}{OT1}{cmr}{m}{sl}
\newenvironment{dedication}
  {
   \thispagestyle{empty}
   \vspace*{0.2cm}
   \itshape             
   \raggedleft          
  }
  {\par 
   \vspace{1cm} 
  }

\newtheorem{lemma}{Lemma}
\newtheorem{prop}{Proposition}
\newtheorem{thm}{Theorem}
\newtheorem{cor}{Corollary}
\theoremstyle{definition}
\newtheorem{defn}{Definition}
\theoremstyle{remark}
\newtheorem{rem}{Remark}

\newcommand{\R}{{\mathbb R}}
\newcommand{\C}{{\mathbb C}}

\newcommand{\T}{{\mathbb T}}

\newcommand{\Sph}{{\mathbb S}}

\newcommand{\cL}{{\mathcal L}}

\newcommand{\Scal}{{\rm Scal}}
\newcommand{\vol}{\mathit{vol}}

\newcommand{\tor}{{\mathfrak t}}

\newcommand{\Hess}{\mathop{\mathrm{Hess}}}

\newcommand{\Pol}{\mathrm P}

 \newcommand{\tstX}{\mathscr{X}}
 \newcommand{\tstL}{\mathscr{L}}

\newcommand{\Fut}{{\rm Fut}}
\newcommand{\Aut}{\mathrm{Aut}}

\newcommand{\proj}{{\mathrmsl proj}}

\newcommand{\Ric}{{\mathrm{Ric}}}

\makeatletter
\@namedef{subjclassname@2020}{\textup{2020} Mathematics Subject Classification}
\makeatother


\begin{document}
\begin{dedication}
Dedicated to Professor T. Mabuchi \\on the occasion of his 75th birthday.  
\end{dedication}

\title[Mabuchi's observation]{Mabuchi K\"ahler solitons versus extremal K\"ahler metrics and beyond}

\author{Vestislav Apostolov}
\address{V.\,Apostolov\\ D{\'e}partement de Math{\'e}matiques\\ Universit\'e du Qu\'ebec \`a Montr\'eal \\
 and \\ Institute of Mathematics and Informatics\\ Bulgarian Academy of Sciences}
\email{\href{mailto:apostolov.vestislav@uqam.ca}{apostolov.vestislav@uqam.ca}}

\author{Abdellah Lahdili}
\address{A.\,Lahdili\\ D{\'e}partement de Math{\'e}matiques\\ Universit\'e du Qu\'ebec \`a Montr\'eal }
\email{\href{mailto:lahdili.abdellah@gmail.com}{lahdili.abdellah@gmail.com}}

\author{Yasufumi Nitta}
\address{Y.\,Nitta \\ Department of Mathematics\\ Faculty of Science Division II\\ Tokyo University of Science}
\email{\href{mailto:nitta@rs.tus.ac.jp}{nitta@rs.tus.ac.jp}}

\thanks{V. Apostolov was supported by an NSERC Discovery grant and a``Connect Talent'' grant of the Region de la Pays de la Loire. A. Lahdili was supported by Aarhus University and the UQAM. Y. Nitta was supported by JSPS KAKENHI Grant Number JP21K03234. The authors thank LMJL of Nantes University and the CIRGET of UQAM for hospitality during the preparation of this work. They are grateful to the referees for careful reading of the manuscript and valuable suggestions. }

\begin{abstract} Using the Yau-Tian-Donaldson type correspondence for $v$-solitons established by Han-Li, we show that a smooth complex $n$-dimensional Fano variety admits a Mabuchi soliton provided it admits an extremal K\"ahler metric whose scalar curvature is strictly less than $2(n+1)$. Combined with  previous observations by  Mabuchi and Nakamura in the other direction, this gives a characterization of  the existence of Mabuchi solitons in terms of the existence of  extremal K\"ahler metrics on Fano manifolds.
An extension of this correspondence to $v$-solitons is also obtained.
\end{abstract}

\maketitle

\section{Introduction} 
A number of different notions of \emph{special} K\"ahler metrics have emerged in the last 20 years or so, in connection with Calabi's seminal program \cite{Calabi55} of finding a canonical representative of a given deRham class $\alpha\in H^2_{dR}(X, \R)$ of K\"ahler metrics on a smooth compact K\"ahler manifold $X$. Lead by a natural variational approach, Calabi himself proposed the notion of an \emph{extremal} K\"ahler metric as a candidate for such a representative.  Recall that a  K\"ahler $\omega \in \alpha$ is extremal if  the flow of the gradient vector of its scalar curvature preserves the complex structure of $X$. K\"ahler metrics with constant scalar curvature (cscK) are a special case of extremal K\"ahler metrics.

In the case of a smooth Fano variety $X$  a cscK metric in $\alpha=2\pi c_1(X)$ is necessarily a K\"ahler--Einstein metric with scalar curvature equal to $2n$. The existence problem for K\"ahler-Einstein metrics is now understood in terms of the Yau--Tian--Donaldson  (YTD) conjecture~\cite{Yau90, Tia97, Do02}  which states that  $X$ admits a K\"ahler-Einstein metric in $2\pi c_1(X)$ if and only if the anticanonical polarization $(X, K^{-1}_X)$ is K-polystable. There are, by now, many different proofs of this conjecture \cite{Tia15, CSW18, Zha23, BBJ21, Li22}, following
the initial work of Chen–Donaldson–Sun \cite{CDS1, CDS2, CDS3} (who proved K-polystability implies
existence) and Tian and Berman (who proved existence implies K-stablity)\cite{Tia97,Berman-Inv}. 

Beyond the study of K\"ahler-Einstein metrics on $(X, 2\pi c_1(X))$, other notions of canonical K\"ahler metrics have been considered. These  allow to treat cases where a K\"ahler-Einstein metric does not exist due to obstructions  in terms of the automorphisms of $X$ \cite{matsushima, licherowicz, Futaki0}. Following initial existence results \cite{Cao,S},  Tian-Zhu~\cite{TZ1,TZ2} initiated a comprehensive study of the so-called \emph{K\"ahler-Ricci solitons} (KRS) on $(X, 2\pi c_1(X))$. Martelli--Sparks--Yau~\cite{MSY} developed the theory of Calabi--Yau cones (or, equivalently, Sasaki--Einstein structures) defined on the affine cone $K_X^{\times}$ associated to $X$. Mabuchi~\cite{Mabuchi} introduced the notion of an \emph{M-soliton}  (or \emph{Mabuchi soliton}) and related it to the existence of extremal K\"ahler metrics in $\alpha = 2\pi c_1(X)$. These works prompted separate investigations of the corresponding existence theories, and the formulation and proofs of appropriate modifications of the YTD conjecture in each case, see respectively \cite{DSz,CSz,Hisamoto}. 

\bigskip
More recently, there have been developments providing a framework to treat the existence problems mentioned above all together: Building on a foundational work by Berman--Witt Nystr\"om~\cite{BN}, Han--Li~\cite{HL} studied the existence problem for the so-called \emph{$v$-soliton} K\"ahler metric $\omega \in 2\pi c_1(X)$ (called $g$-solitons in \cite{HL}). A $v$-soliton  is defined in terms of a fixed maximal compact torus $\T\subset \Aut(X)$ with associated canonical polytope $\Pol_X \subset \left({\rm Lie}(\T)\right)^*$,  and  a positive smooth function $v(x)>0$ on $\Pol_X$, via the equation
\[ {\rm Ric}(\omega) - \omega = \frac{1}{2}dd^c \log v(\mu_{\omega}).\]
In the above formula, $\omega$ is a $\T$-invariant K\"ahler metric in $2\pi c_1(X)$,  ${\rm Ric}(\omega)\in 2\pi c_1(X)$ is its Ricci form and $\mu_\omega : X \to \Pol_X$ is the canonically normalized $\T$-momentum map. Thus, K\"ahler--Einstein metrics correspond to $1$-solitons, KRS metrics to $e^{\ell}$-solitons~\cite{TZ1},  Calabi--Yau cone structures on $K_X^{\times}$ to $\ell^{-(n+2)}$-solitons~\cite{AJL} and Mabuchi solitons to $\ell$-solitons, where $\ell(x)$ is a suitably defined (and in general different for each  case) affine-linear function on $\Pol_X$. As an outcome, the work \cite{HL} gives a YTD type correspondence for the existence of a $v$-soliton on $(X, \T, 2\pi c_1(X))$, expressed in terms of a suitable notion of \emph{uniform $v$-weighted Ding stability}  of $(X, \T, 2\pi c_1(X)))$ on $\T_{\C}$-equivariant test-configurations.

\bigskip
In this short paper, we use the Han-Li YTD type correspondence for $v$-solitons combined with a few elementary observations in order to establish the following result. 

\begin{thm}\label{main}[see Theorem~\ref{(1,w)-YTD}] Let $(X, \T)$ be a smooth complex $n$-dimensional Fano manifold and $v$ a smooth positive weight function on $\Pol_X$ normalized so that  $\vol_v(X)= \vol(X)$. Then  $X$  admits a  $\T$-invariant $v$-soliton metric in $2\pi c_1(X)$ if and only if there exists a (different in general) $\T$-invariant K\"ahler metric $\omega \in 2\pi c_1(X)$ whose scalar curvature satisfies
\[ \Scal(\omega) = 2\Big((n+1)-v(\mu_{\omega})\Big).\]
\end{thm}

Futaki-Mabuchi~\cite{FM} observed that there exists a unique affine-linear function $\ell_{\rm ext}(x)$ on $\Pol_X$  such that a $\T$-invariant K\"ahler metric in $2\pi c_1(X)$ is extremal iff $\Scal(\omega) = \ell_{\rm ext}(\mu_{\omega})$. Furthermore, Mabuchi~\cite{Mabuchi} proved that if $\tilde \omega \in 2\pi c_1(X)$ is an $\ell$-soliton for a positive affine-linear function $\ell$ on $\Pol_X$, then 
\[ \ell = (n+1) - \frac{1}{2} \ell_{\rm ext}. \]
 He thus obtained that the condition 
\begin{equation}\label{Mabuchi-constant}
 \max_{\Pol_X} \ell_{\rm ext}  < 2(n+1) \end{equation}
is a necessary condition for the existence of an $\ell$-soliton. On the other hand, Nakamura (see~\cite{Mabuchi}) proved,  using the result  of Chen-Cheng~\cite{CC} and the complement by  He~\cite{He}  that if $X$ admits a Mabuchi soliton then it also admits an extremal K\"ahler metric satisfying the above condition. 
This observation was recently extended for the weights corresponding to the so-called \emph{$\sigma$-solitons} by Nakagawa and Nakamura (\cite{NN24}). 
Finally, in \cite{NSY23, NS24}, the authors constructed examples of smooth Fano manifolds admitting extremal K\"ahler metrics in $2\pi c_1(X)$ for which the necessary condition \eqref{Mabuchi-constant} fails (and thus they do not admit Mabuchi solitons). As a direct corollary of  Theorem~\ref{main} above, we close this circle of ideas by establishing
\begin{cor} A smooth Fano manifold $(X, \T)$ admits a Mabuchi soliton if and only if the extremal affine linear function $\ell_{\rm ext}$ satisfies \eqref{Mabuchi-constant} and $2\pi c_1(X)$ admits an extremal metric.
\end{cor} 

\section{Preliminaries} 
\subsection{Fano manifolds: notation and normalization}
In what follows, $X$ denotes a smooth compact complex manifold of complex dimension $n$, for which the anti-canonical bundle $K^{-1}_X$ is ample. Such an $X$ is  called a \emph{smooth Fano variety}. The Fano condition implies that $X$ is projective, and that the deRham class $\alpha=2\pi c_1(X)=2\pi c_1(K_X^{-1})$ contains K\"ahler metrics. 

Any K\"ahler metric $\omega \in \alpha$ is deRham cohomologous with the corresponding Ricci form $\Ric(\omega)$, and thus we can write 
\begin{equation}\label{Ricci-potential} \Ric(\omega) - \omega = \frac{1}{2} dd^c h_{\omega},\end{equation}
for a smooth function $h_{\omega}$ which (by the maximum principle) is unique up to an additive constant. Such a function will be referred to as a \emph{Ricci potential} of $\omega$;
we can further fix the additive constant  
by requiring that
\begin{equation}\label{normalized-Ricci}\int_X e^{\mathring{h}_{\omega}} \omega^{[n]} = \int_X \omega^{[n]}=:\vol(X),\end{equation}
where $\omega^{[n]}:= \omega^n/n!$ stands for the Riemannian volume form of the K\"ahler metric $\omega$. We shall  then refer to this uniquely defined Ricci potential $\mathring{h}_{\omega}$ as the \emph{normalized Ricci potential} of $\omega$.
In these terms, the \emph{K\"ahler--Einstein} condition 
\begin{equation}\label{eq:KE} \Ric(\omega)=\omega \end{equation}
is equivalent to $\mathring{h}_{\omega}=0$.

\bigskip
We shall next fix once for all a maximal compact real torus $\T$ inside the connected component of identity $\Aut_0(X)$ of the group of complex automorphisms of $X$.  The corresponding complex torus will be denoted by $\T_{\C}$. There is a canonical lift (still denoted by $\T$) of the action of $\T$ on $X$ to  an action on the canonical bundle $K_X$. The latter bundle has a further ${\Sph}^1$-extension of the lifted  $\T$ action, given  by fibre-wise multiplications with complex numbers $e^{i \theta}\in \Sph^1$. We denote by $\hat \T= \T \times \Sph^1$  the resulting $({\rm dim}(\T) +1)$-dimensional torus acting on $K_X$.
We shall respectively denote by $\tor$ and $\hat \tor$ the Lie algebras of $\T$ and $\hat{\T}$. 

We consider the space
$\mathcal{K}_{\alpha}^{\T}(X)$ of $\T$-invariant K\"ahler metrics $\omega$ on $X$, belonging to $\alpha$; by  a standard averaging argument, $\mathcal{K}_{\alpha}^{\T}(X) \neq\emptyset$. Introducing a base-point $\omega_0\in \mathcal{K}_{\alpha}^{\T}(X)$, we will identify $\mathcal{K}_{\alpha}^{\T}(X)$ with the Fr\'echet space ${\mathcal H}^{\T}_{\omega_0}(X)/\R$, where 
\[\mathcal{H}^{\T}_{\omega_0}(X):= \left\{ \varphi \in C^{\infty}(X)^{\T} \, \, | \, \, \omega_{\varphi}:= \omega_0 + dd^c\varphi >0\right\}\]
is the space of smooth $\T$-invariant K\"ahler potentials with respect to $\omega_0$.

For each $\omega \in \mathcal{K}_{\alpha}^{\T}(X)$, we let $H_{\omega}$ denote the Hermitian metric on $K_X$ whose Chern curvature is
$R^{H_{\omega}}= -i \omega$,  and by $\nabla^{H_\omega}$ the $H_{\omega}$-Chern connection on $K_X$. A basic fact in the theory~\cite{gauduchon-book} is that for any $\xi \in \tor$, and any smooth section $s\in C^{\infty}(X, K_X)$, we have 
\[ {\mathcal L}_{\xi} s = \nabla^{H_\omega}_{\xi} s -i \mu_{\omega}^{\xi} s,\]
where the smooth function $\mu_{\omega}^{\xi}$ satisfies
\[ \omega(\xi, \cdot) = -d\mu_{\omega}^{\xi}.\]
This gives rise to a \emph{canonically normalized} momentum map $\mu_{\omega} : X \to \tor^{*}$ whose image $\Pol_X$ is a compact convex polytope~\cite{Atiyah, GS}; one can further show (see e.g.~\cite{BN,lahdili}) that $\mu_{\omega}(\Pol_X)$ is independent of the choice of $\omega\in \mathcal{K}^{\T}_{\alpha}(X)$.  In this paper, we shall refer to $\Pol_X$ as the \emph{canonical polytope} of $(X,\T)$. 

\begin{rem}\label{r:canonical_normalization} In general, the $\T$-momentum map $\mu_{\omega} : X \to \tor^*$ is defined only up to a translation with an element of $\tor^*$;  the fact that in the Fano case there is  a canonical normalization for $\mu_{\omega}$ follows from the existence of a canonical lift of the $\T$-action on $X$ to $K_X$. 

An alternative way to define the canonical normalization for $\mu_{\omega} $(see e.g. \cite{AJL,TZ1}) is to require that  for any $\zeta\in \tor$, the function $\mu_{\omega}^{\zeta}:= \langle \mu_{\omega}, \zeta \rangle$ satisfies
\begin{equation}\label{normalized-mu}\int_X \mu_{\omega}^{\zeta} e^{h_{\omega}} \omega^{[n]} =0, \end{equation}
where $h_{\omega}$ is any Ricci potential of $\omega$.
\end{rem}
Once we have suitably normalized $\Pol_X$, we can define the \emph{Dustermaat-Heckman}
measure~\cite{DH} $d\mu_{\rm DH}$ on $\Pol_X$ as the push-forward via $\mu_{\omega}$ of the Riemannian measure of $(X, \omega)$: for any continuous function $f(x)$ on $\Pol_X$, we let
\begin{equation}\label{DH}
\int_{\Pol_X} f(x) d\mu_{\rm DH} := \int_X f(\mu_{\omega}) \omega^{[n]}.\end{equation}
The fact that the LHS is independent of the choice of $\omega\in \mathcal{K}^{\T}_\alpha(X)$ follows for instance from the $\T$-equivariant Moser's lemma (see also \cite{FM}).

\subsection{K\"ahler--Ricci solitons}\label{s:KRsoliton} Following \cite{TZ1}, a \emph{K\"ahler Ricci soliton} (KRS for short) is a K\"ahler metric $\omega \in 2\pi c_1(X)$ which satisfies 
\begin{equation}\label{KRS}
{\rm Ric}({\omega}) - \omega = -\frac{1}{2} \cL_{J\tau} \omega,
\end{equation}
where $\tau$ is a Killing vector field for the K\"ahler structure $\omega$. In the case $\tau=0$, \eqref{KRS} reduces to the K\"ahler--Einstein condition \eqref{eq:KE}. Tian--Zhu~\cite{TZ1} have extended the Matsushima's theorem to the case of a KRS, which in turn yields that any K\"ahler metric satisfying \eqref{KRS} must be invariant by the action of a maximal torus in $\Aut_0(X)$, containing the flow of $\tau$. Up to a pull back by an element of $\Aut_0(X)$, we can and will assume that a KRS on $X$ belongs to $\mathcal{K}^{\T}_{\alpha}(X)$ and $\tau \in \tor$. Thus, similarly to the K\"ahler--Einstein case, the KRS condition can be rewritten as 
\begin{equation}\label{KRS-h}
h_{\omega} = \mu_{\omega}^{\tau}\end{equation} or, equivalently,  
\begin{equation}\label{KRS-gradient}
{\rm Ric}({\omega})  -\omega = \frac{1}{2} dd^c \mu_{\omega}^{\tau}, \qquad \tau \in \tor.
\end{equation}
By Remark~\ref{r:canonical_normalization}, \eqref{DH} and \eqref{KRS-h}, if $X$ admits a KRS in $\mathcal{K}^{\T}_\alpha(X)$, then for any $\zeta\in \tor$, we have
\[\int_{\Pol_X} \langle \zeta, x \rangle e^{\langle \tau, x\rangle}d\mu_{\rm DH}=0.\]
The above condition means that $\tau$ is a critical point of the function $F: \tor \to \R$: \[ F(\zeta) := \int_{\Pol_X} e^{\langle \zeta, x\rangle}d\mu_{\rm DH}.\] Tian--Zhu~\cite{TZ1} further show that $F$ admits a unique critical point,  $\tau$, independent of the existence of a KRS on $X$. We shall refer to $\tau$ as the \emph{KRS vector field} of $(X, \T)$ and to the positive smooth function $v(x):=e^{\langle \tau, x \rangle}$ on $\tor^*$ as the \emph{KRS weight function}.

\subsection{$v$-solitons} The notion of KRS extends to the following more general geometric situation, studied by Berman--Witt Nystr\"om in \cite{BN} and,  more recently,  by Han--Li in \cite{HL}. We follow the notation of \cite[Sect.2]{AJL}.
\begin{defn}[$v$-soliton] In the setup as above, let $v(x)$ be a given positive function defined on $\Pol_X$. A K\"ahler metric $\omega \in \mathcal{K}^{\T}_{\alpha}(X)$ is called a \emph{$v$-soliton} if it satisfies
\begin{equation}\label{v-soliton}\Ric(\omega)-\omega = \frac{1}{2}dd^c \log v(\mu_{\omega}).\end{equation}
\end{defn}
Clearly, K\"ahler--Einstein metrics are $1$-solitons whereas KRS are $v=e^{\langle \tau, x\rangle}$-solitons.
Notice that if $\omega$ is a $v$-soliton it is also a $\lambda v$-soliton for any $\lambda>0$. To get rid of this ambiguity,  we shall sometimes consider \emph{normalized} weight functions $\mathring{v}:= \left(\frac{\int_{\Pol_X} d\mu_{\rm DH}}{\int_{\Pol_X} v d\mu_{\rm DH}}\right) v$,  i.e. such that
\begin{equation}\label{v-normalization} \vol_{\mathring{v}} (X):=\int_{X} \mathring{v}(\mu_{\omega}) \omega^{[n]}= \int_X \omega^{[n]} =: \vol(X).\end{equation}
We also notice that for any $v$-soliton, $h_{\omega}=\log(v(\mu_{\omega}))$, so by Remark~\ref{r:canonical_normalization},
the linear function
\begin{equation}\label{v-Futaki} \Fut_{v} : \tor \to \R, \qquad \Fut_v(\zeta):= \int_{\Pol_X} \langle \zeta, x\rangle v(x) d\mu_{\rm DH} \end{equation}
identically vanishes.
\begin{defn}[$v$-Futaki invariant] The linear function defined by \eqref{v-Futaki} is called the \emph{$v$-Futaki invariant} of $(X, \T)$. 
\end{defn}
We next define a functional ${\bf I}_{v}$ on the 
space $\mathcal{H}^{\T}_{\omega_0}(X)$ of $\T$-invariant K\"ahler potentials (see \cite{lahdili,HL}):
\begin{equation}\label{I_v}
 d_{\varphi} {\bf I}_v(\dot \varphi) = \int_X v(\mu_{\omega_{\varphi}}) \dot \varphi \omega_{\varphi}^{[n]}, \qquad {\bf I}_v(0)=0.\end{equation}
Following \cite{HL}, we introduce
\begin{defn}[$v$-Ding functional]\label{d-v-ding} The $v$-Ding functional is the map
${\bf D}_v : {\mathcal H}^{\T}_{\omega_0}(X) \to \R$ given by
\[
{\bf D}_v (\varphi) := - \left(\frac{{\bf I}_v(\varphi)}{\vol_v(X)}\right)- \frac{1}{2}\log\left(\int_M e^{\mathring{h}_{\omega_0}-2\varphi}\frac{\omega_0^{[n]}}{\vol(X)}\right),
\]
where we have set $\vol_v(X):=\int_{\Pol_X} v(x) d\mu_{\rm DH}= \int_X v(\mu_{\omega})\omega^{[n]}$ and $\mathring{h}_{\omega_0}$ stands for the normalized Ricci potential of the base point $\omega_0$, see \eqref{normalized-Ricci}.
\end{defn}
Notice that ${\bf D}_{v}$ does not change if we add a constant to $\varphi$, so it actually descends to a functional, denoted ${\bf D}_v(\omega_{\varphi})$, on the space $\mathcal{H}^{\T}_{\omega_0}(X)/\R \cong \mathcal{K}^{\T}_{\alpha}(X)$.

It is not hard to see that the differential of ${\bf D}_v$ is given by
\begin{equation}\label{d-Dv}
(d_{\omega_\varphi}{\bf D}_v) (\dot\varphi) = \int_X \dot{\varphi} \left( \frac{e^{\mathring{h}_{\omega_\varphi}}}{\vol(X)}-\frac{v(\mu_{\omega_\varphi})}{\vol_{v}(X)}\right)\omega^{[n]}_\varphi,
\end{equation}
so that the critical points of ${\bf D}_v$ are precisely the K\"ahler metrics $\omega_{\varphi}$ for which
\[e^{\mathring{h}_{\omega_{\varphi}}} = \vol(X)\mathring{v}(\mu_{\omega_{\varphi}})\]
i.e. the $v$-solitons.

Another consequence of the formula \eqref{d-Dv} is the following
\begin{lemma}\label{TC-invariance} The $v$-Ding functional is $\T_{\C}$-invariant, i.e. satisfies 
\[{\bf D}_v(\sigma^*(\omega)) = {\bf D}_v(\omega) \qquad \forall \sigma \in \T_{\C},\] iff the $v$-Futaki invariant $\Fut_v \equiv 0$.
\end{lemma}
\begin{proof} ${\bf D}_v$ is clearly $\T$ invariant. For any $\zeta \in \tor$, we consider the flow of $-J\zeta$ $\sigma_t\in \T_{\C}$,   and take derivative at $t=0$ of ${\bf D}_v(\sigma_t^*(\omega))$. By \eqref{d-Dv} 
\[\frac{d}{dt}_{|_{t=0}}{\bf D}_v(\sigma_t^*(\omega))=\int_X \mu_{\omega}^{\zeta} \left( \frac{e^{\mathring{h}_{\omega}}}{\vol(X)}-\frac{v(\mu_{\omega})}{\vol_{v}(X)}\right)\omega^{[n]}= -\Fut_v(\zeta),\]
where we have used 
\eqref{normalized-mu} for the canonically normalized momentum map $\mu_{\omega}$. The claim follows from the above by a standard argument.
\end{proof}
We end-up this section with stating one of the main results of \cite{HL} (see Theorem 1.6), which gives an anlaytic criterion for the existence of a $v$-soliton on $(X, \T)$ in terms of ${\bf D}_v$. Recall the definition~\cite{Aubin} of the Aubin  functional ${\bf J} : \mathcal{H}^{\T}_{\omega_0}(X) \to \R$:
\[ {\bf J}(\varphi) := \int_X \varphi \omega_0^{[n]}-{\bf I}_1(\varphi). \]
The functional ${\bf J}$ descends to ${\mathcal H}^{\T}(X)/\R$ (so we denote it by ${\bf J}(\omega_{\varphi})$), and has the property that ${\bf J}(\omega_{\varphi})\ge 0$ with ${\bf J}(\omega_{\phi})=0$ iff $\omega_{\varphi}=\omega_0$. 
\begin{thm}\cite{HL}\label{thm:HL} Let $X$ be a Fano manifold,  $\T\subset\Aut_0(X)$ a maximal compact torus with canonical momentum polytope $\Pol_X\subset \tor^*$ and $v$ a positive smooth function on $\Pol_X$. Then $X$ admits a $\T$-invariant $v$-soliton in $2\pi c_1(X)$ if and only if the $v$-weighted Ding functional ${\bf D}_v$ is invariant and coercive with respect to the complex torus $\T_{\C}$, i.e. there exist constants $\Lambda>0$ and  $C$ such that for any  $\omega \in {\mathcal K}^{\T}_\alpha(X)$, 
\[{\bf D}_{v}(\omega) \ge \Lambda \inf_{\sigma \in \T_{\C}} {\bf J}(\sigma^*(\omega)) - C. \]
\end{thm}
Versions of the above theorem have been known for KRS by the works of Cao--Tian--Zhu (see~\cite[Theorems 0.1 and 0.2]{CTZ}) and Darvas--Rubinstein (see \cite[Theorem 2.11]{DR}).
\section{$v$-solitons as weighted constant scalar curvature metrics} In \cite[Prop.1]{AJL}, it is observed that any $v$-soliton $\omega \in 2\pi c_1(X)$ can be equivalently described as a $\T$-invariant K\"ahler metric in $2\pi c_1(X)$ with $(v, \tilde v)$-constant scalar curvature in sense of Lahdili~\cite{lahdili}, where  the smooth function ${\tilde v}(x)$ on $\Pol_X$  is defined by
\begin{equation}\label{tilde-v}
{\tilde v}(x): = 2\left(n + \sum_{i=1}^{\ell}(\log v)_{, i} x_i\right)v(x).\end{equation}

To recall the construction in \cite{AJL}, we consider the notion of  a $(v, w)$-cscK metric~\cite{lahdili}, introduced by the equation
\begin{equation}\label{v-scal}
\Scal_v(\omega):= v(\mu_{\omega}) \Scal(\omega) + 2\Delta_{\omega} v(\mu_{\omega})  + \big\langle g_{\omega}, \mu_{\omega}^*\left(\Hess(v)\right)\big\rangle = w(\mu_{\omega}), \end{equation}
where $g_{\omega}$  is the riemannian metric associated to $\omega$ and   the contraction $\langle \cdot, \cdot \rangle$  is taken between the smooth $\tor^*\otimes \tor^*$-valued function $g_{\omega}$ on $X$  (the restriction of the riemannian metric $g_{\omega}$  to $\tor \subset C^{\infty}(X, TX)$)  and the smooth $\tor\otimes \tor$-valued function ${\mu_{\omega}}^*\left(\Hess(v)\right)$ on $X$ (given by the pull-back by $\mu_{\omega}$ of $\Hess(v) \in C^{\infty}(\Pol_X, \tor\otimes \tor)$); equivalently,  if $\{\xi_i\}_{i=1, \ldots \ell}$ is a basis of $\tor$, we have
\[ \big\langle g_{\omega}, \mu_{\omega}^*\left(\Hess(v)\right)\big\rangle = \sum_{i,j=1}^{\ell}v_{,ij}(\mu_{\omega})g_{\omega}(\xi_i, \xi_j).\]
In these terms, we have 
\begin{prop}\cite{AJL}\label{p:weights}  $\omega \in \mathcal{K}_{2\pi c_1}^{\T}(X)$ is a $v$-soliton if and only if it is a $(v, \tilde v)$-cscK metric, i.e. satisfies \eqref{v-scal} with $w=\tilde v$.
\end{prop}
Notice that when $v=1$, we recover the basic fact that the K\"ahler--Einstein metrics in $2\pi c_1(X)$ are the K\"ahler metrics of constant scalar curvature equal to $2n$.

\smallskip
The advantage of this point of view is that it leads to a weighted version of  the Mabuchi energy, which we shall denote by  ${\bf M}_{v, w}(\omega_{\varphi})$,  defined on $\mathcal{H}_{\omega_0}^{\T}(X)/\R$ by (see \cite{lahdili}):
\[ (d_{\omega_\varphi} {\bf M}_{v, w})(\dot \varphi):= -\int_X\dot{\varphi} \left(\Scal_v(\omega_{\varphi})-w(\mu_{\varphi})\right)\omega_\varphi^{[n]}, \qquad {\bf M}_{v, w}(\omega_0)=0.\]
The  critical points of ${\bf M}_{v,w}$ are the $\T$-invariant $(v,w)$-cscK metrics in a given K\"ahler class $\alpha=[\omega_0]$, and  by Proposition~\ref{p:weights}, in the special case when $\alpha = 2\pi c_1(X)$ and $w=\tilde v$, the critical points of ${\bf M}_{v, \tilde v}$ are the $v$-solitons.  We also have  a Futaki type invariant obstructing the existence of a  $(v, w)$-cscK metric  (see \cite{lahdili}):
\[ \Fut_{v, w} : \tor \to \R, \qquad  \Fut_{v, w}(\zeta): =  - \int_X \mu_{\omega}^{\zeta} \left(\Scal_v(\omega)-w(\mu_{\omega})\right) \omega^{[n]}, \]
which is independent of $\omega \in \mathcal{K}_{\alpha}^{\T}(X)$ and thus vanishes should a $(v,w)$-cscK metric exist in $\alpha$.  It is clear from the definitions that ${\bf M}_{v, w}$  is invariant under the natural action of $\T_{\C}$ on ${\mathcal K}_{\alpha}^{\T}(X)$ by pull-backs  if and only if $\Fut_{v,w} =0$; furthermore
\begin{thm}\cite{AJL} If  $(X, \T, \alpha)$ admits a $(v,w)$-cscK metric, then ${\bf M}_{v,w}$ is coercive relative to $\T_\C$, i.e. ${\bf M}_{v,w}$ is invariant under the action of $\T_{\C}$ on ${\mathcal K}_{\alpha}^{\T}(X)$ and there exist constants $\Lambda>0, C$ such that for any $\omega \in {\mathcal K}_{\alpha}^{\T}(X)$
\[ {\bf M}_{v,w}(\omega) \ge \Lambda \inf_{\sigma\in \T_{\C}}{\bf J}(\sigma^*(\omega)) - C.\] \end{thm}
In  \cite[Thm.~5]{lahdili}, it is shown that ${\bf M}_{v, w}$ admits a Chen--Tian decomposition  as the sum of  energy and entropy parts:
\begin{align}
\begin{split}\label{Chen-Tian}
{\bf M}_{v, w}(\omega_\varphi)=\int_X\log\left(\frac{v(\mu_\varphi)\omega^{[n]}_\varphi}{\omega_0^{[n]}}\right) v(\mu_\varphi)\omega^{[n]}_\varphi &-2{\bf I}^{\Ric(\omega_0)}_{v}(\varphi)+{\bf I}_w(\varphi) \\
&- \int_X \log(v(\mu_0))v(\mu_0)\omega_0^{[n]},
\end{split}
\end{align}
where the functional ${\bf I}_w: \mathcal{H}_{\omega_0}^{\T}(X) \to\R$ is defined in \eqref{I_v}
and for a fixed $\T$-invariant closed $(1,1)$-form $\rho$ on $X$ with momentum $\mu_\rho: X\to\tor^{*}$,  the functional ${\bf I}^{\rho}_{v}: \mathcal{H}_{\omega_0}^{\T}(X)  \to\R$ is defined by
\begin{equation*}\label{Energi-chi}
(d_\varphi{\bf I}^{\rho}_{v})(\dot{\varphi}):=\int_X \dot{\varphi}\left(v(\mu_\varphi)\rho\wedge\omega_{\varphi}^{[n-1]}+\langle (dv)(\mu_\varphi),\mu_\rho\rangle\omega_{\varphi}^{[n]} \right),\quad {\bf I}^{\rho}_{v}(0)=0.
\end{equation*}
In \cite{HL} (see Definition 2.10 and  the second identity in the proof of 
Lemma 2.12 in that reference), the authors introduce a weighted version of the Mabuchi functional whose critical points are $v$-solitons. Seeing $v$-solitons as $(v, \tilde v)$-cscK metrics, it is not immediately clear that Han--Li's weighted Mabuchi functional equals 
 ${\bf M}_{v, \tilde v}$,  up to an additive constant. We check this in the lemma below, where in our notation the first line of \eqref{eq:Mvtv}  corresponds to Han-Li's weighted Mabuchi functional.
\begin{lemma}\label{Mab-Fano}
On a Fano manifold $(X,\T,\alpha=2\pi c_1(X))$,  the weighted Mabuchi energy ${\bf M}_{v, \tilde v}$ corresponding  to $v$-solitons has the following equivalent expression
\begin{equation}
\begin{split}\label{eq:Mvtv}
{\bf M}_{v, \tilde v}(\omega_\varphi)=&\int_X\log\left(\frac{v(\mu_\varphi)\omega^{[n]}_\varphi}{e^{h_0}\omega_0^{[n]}}\right) v(\mu_\varphi)\omega^{[n]}_\varphi +2\left(\int_X\varphi v(\mu_\varphi)\omega_\varphi^{[n]} - {\bf I}_{v}(\varphi)\right) \\
&- \int_X \log(e^{-h_0}v(\mu_0))v(\mu_0)\omega_0^{[n]},
\end{split}
\end{equation}
where $h_0$ is a Ricci potential of $\omega_0$.
\end{lemma}
\begin{proof}
To simplify notation, we index by $\varphi$ the geometric quantities depending on the K\"ahler structure $\omega_{\varphi}$. We fix a basis of $\mathbb{S}^{1}$ generators $\{\xi_j\}_{j=1,\cdots,\ell}$ of $\tor$.  Using the relations $\Ric(\omega_{0})=\omega_0+\frac{1}{2}dd^{c}h_0$,  we have that the momentum map with respect to $\Ric(\omega_0)$ is 
\[\mu_{\Ric(\omega_{0})}^{\zeta}= \mu^{\zeta}_{0}-\frac{1}{2}\mathcal{L}_{J\zeta}h_0, \qquad \zeta\in\tor.\]
We further normalize  $h_0$ by $\int_X e^{h_{0}} \omega_{0}^{[n]} =1$  and compute
\[
\begin{split}
(d_\varphi{\bf I}^{\Ric(\omega_0}_{v})(\dot{\varphi})=&\int_X \dot{\varphi}\Big(v(\mu_\varphi)\Ric(\omega_0)\wedge\omega_{\varphi}^{[n-1]}+\sum_{j=1}^{\ell} v_{,j}(\mu_\varphi)\mu_{\Ric(\omega_0)}^{\xi_j}\omega_{\varphi}^{[n]} \Big)\\
=&\frac{1}{2}\int_X \dot{\varphi}\Big(-v(\mu_\varphi)\Delta_{\varphi}(h_0)\omega_{\varphi}^{[n]}-\sum_{j=1}^{\ell} v_{,j}(\mu_\varphi)(\mathcal{L}_{J\xi_j}h_0)\omega_{\varphi}^{[n]} \Big)
+(d_\varphi{\bf I}^{\omega_0}_{v})(\dot{\varphi})\\
=&\frac{1}{2}\int_X \dot{\varphi}\Big(-\langle d(v(\mu_\varphi)),dh_0\rangle_{\omega_\varphi} \omega_{\varphi}^{[n]}-\sum_{j=1}^{\ell} v_{,j}(\mu_\varphi)(\mathcal{L}_{J\xi_j}h_0)\omega_{\varphi}^{[n]} \Big)\\
&-\frac{1}{2}\int_X\langle dh_0,d\dot{\varphi}\rangle_\varphi v(\mu_\varphi)\omega_\varphi^{[n]}+(d_\varphi{\bf I}^{\omega_0}_{v})(\dot{\varphi})\\
=&\frac{1}{2}\int_X \dot{\varphi}\Big(-\sum_{j=1}^{\ell} v_{,j}(\mu_\varphi)\langle d\mu^{\xi_j}_\varphi,dh_0\rangle_{\omega_\varphi} \omega_{\varphi}^{[n]}-\sum_{j=1}^{\ell} v_{,j}(\mu_\varphi)(\mathcal{L}_{J\xi_j}h_0)\omega_{\varphi}^{[n]} \Big)\\
&-\frac{1}{2}\int_X\langle dh_0,d\dot{\varphi}\rangle_\varphi v(\mu_\varphi)\omega_\varphi^{[n]}+(d_\varphi{\bf I}^{\omega_0}_{v})(\dot{\varphi})\\
=&-\frac{1}{2}\int_X\langle dh_0,d\dot{\varphi}\rangle_\varphi v(\mu_\varphi)\omega_\varphi^{[n]}+(d_\varphi{\bf I}^{\omega_0}_{v})(\dot{\varphi}),
\end{split}
\]
where we used the identity $\langle d\mu^{\xi_j}_\varphi,dh_0\rangle_{\omega_\varphi}=-\mathcal{L}_{J\xi_j}h_0$.  On the other hand, noting that
\begin{equation}\label{v-variation} 
d_\varphi[v(\mu_{\varphi})](\dot{\varphi})=\sum_{i=1}^{\ell} v_{,i}(\mu_{\varphi})(d^{c}\dot\varphi)(\xi_i)=\langle d[ v(\mu_{\varphi})], d\dot\varphi\rangle_\varphi,
\end{equation}
we have
\[
\begin{split}
d_\varphi\left[\int_X h_0v(\mu_\varphi)\omega_\varphi^{[n]}\right]=&\int_X h_0\langle d[ v(\mu_{\varphi})], d\dot\varphi\rangle_\varphi\omega_\varphi^{[n]}-\int_X h_0v(\mu_\varphi)\Delta_\varphi(\dot{\varphi})\omega_\varphi^{[m]}\\
=&-\int_X\langle dh_0,d\dot{\varphi}\rangle_\varphi v(\mu_\varphi)\omega_\varphi^{[n]}
\end{split}
\]
and hence
\begin{equation}\label{I_Fano1}
{\bf I}^{\Ric(\omega_0)}_{v}(\varphi)={\bf I}^{\omega_0}_{v}(\varphi)+\frac{1}{2}\int_X h_0v(\mu_\varphi)\omega_\varphi^{[n]}-\frac{1}{2}\int_X h_0v(\mu_0)\omega_0^{[n]}.
\end{equation}
The last integral in the RHS arises from the normalization ${\bf I}_v^{\Ric(\omega_0)}(0)={\bf I}_v^{\omega_0}(0)=0$. Using  $\mu^{\zeta}_0=\mu^{\zeta}_\varphi+\mathcal{L}_{J\zeta}\varphi$ for $\zeta\in \tor$, we compute further
\[
\begin{split}
(d_\varphi{\bf I}^{\omega_0}_{v})(\dot{\varphi})=&\int_X \dot{\varphi}\Big(v(\mu_\varphi)\omega_0\wedge\omega_{\varphi}^{[n-1]}+\sum_{j=1}^{\ell} v_{,j}(\mu_\varphi)\mu_{\omega_0}^{\xi_j}\omega_{\varphi}^{[n]} \Big)\\
=&\int_X \dot{\varphi}\Big(v(\mu_\varphi)(\omega_\varphi-dd^{c}\varphi)\wedge\omega_{\varphi}^{[n-1]}+\sum_{j=1}^{\ell} v_{,j}(\mu_\varphi)(\mu_{\varphi}^{\xi_j}+\mathcal{L}_{J\xi_j}\varphi)\omega_{\varphi}^{[m]} \Big).
\end{split}\]
Integrating by parts further leads to
\[
\begin{split}
(d_\varphi{\bf I}^{\omega_0}_{v})(\dot{\varphi})=& n (d_\varphi{\bf I}_{v})(\dot{\varphi}) + \int_X \dot{\varphi}\langle d(v(\mu_\varphi)),d\varphi\rangle_\varphi\omega_{\varphi}^{[n]}+\int_X \langle d\dot{\varphi},d\varphi\rangle_\varphi v(\mu_\varphi)\omega_{\varphi}^{[n]}\\
&+\int_X \dot{\varphi}\sum_{j=1}^{\ell} v_{,j}(\mu_\varphi)(\mu_{\varphi}^{\xi_j}+\mathcal{L}_{J\xi_j}\varphi)\omega_{\varphi}^{[n]}\\
=&n(d_\varphi{\bf I}_{v})(\dot{\varphi})-\int_X \dot{\varphi} \sum_{j=1}^{\ell}v_{,j}(\mu_\varphi)(\mathcal{L}_{J\xi_j}\varphi)\omega_{\varphi}^{[n]}+\int_X \langle d\dot{\varphi},d\varphi\rangle_\varphi v(\mu_\varphi)\omega_{\varphi}^{[m]}\\
&+\int_X \dot{\varphi}\sum_{j=1}^{\ell} v_{,j}(\mu_\varphi)(\mu_{\varphi}^{\xi_j}+\mathcal{L}_{J\xi_j}\varphi)\omega_{\varphi}^{[n]} \\
=&n(d_\varphi{\bf I}_{v})(\dot{\varphi})+\int_X \langle d\dot{\varphi},d\varphi\rangle_\varphi v(\mu_\varphi)\omega_{\varphi}^{[n]} +\int_X \dot{\varphi}\langle (dv)(\mu_\varphi),\mu_{\varphi}\rangle\omega_{\varphi}^{[n]} \\
=&n(d_\varphi{\bf I}_v)(\dot{\varphi})+\int_X \langle d\dot{\varphi},d\varphi\rangle_\varphi v(\mu_\varphi)\omega_{\varphi}^{[n]}+\int_X \dot{\varphi}\left(\frac{1}{2}\tilde{v}(\mu_\varphi)-n v(\mu_\varphi)\right)\omega_\varphi^{[n]}\\
=&\int_X \langle d\dot{\varphi},d\varphi\rangle_\varphi v(\mu_\varphi)\omega_{\varphi}^{[n]}+\frac{1}{2}(d_\varphi{\bf I}_{\tilde{v}})(\dot{\varphi}). 
\end{split}
\]
By \eqref{v-variation} we have
\[
\int_X \langle d\dot{\varphi},d\varphi\rangle_\varphi v(\mu_\varphi)\omega_{\varphi}^{[n]} =-d_\varphi\left[\int_X \varphi v(\mu_\varphi)\omega_\varphi^{[n]}\right]+(d_\varphi{\bf I}_v)(\dot{\varphi}),
\]
and hence
\begin{equation}\label{I_Fano2}
{\bf I}^{\omega_0}_{v}(\varphi)={\bf I}_v(\dot{\varphi})+\frac{1}{2}{\bf I}_{\tilde v}(\varphi)-\int_X\varphi v(\mu_\varphi)\omega_\varphi^{[n]}. 
\end{equation}
Substituting \eqref{I_Fano1} and \eqref{I_Fano2} in \eqref{Chen-Tian} we obtain the formula \eqref{eq:Mvtv}. \end{proof}

Using Lemma~\ref{Mab-Fano}, we obtain
\begin{lemma}\label{M_v/D_v}
Let $(X,\T,\alpha=2\pi c_1(X))$ be a Fano manifold and $v>0$ a smooth positive weight function on $\Pol_X$.  Then the weighted Ding functional ${\bf D}_v$ is related to the weighted Mabuchi energy ${\bf M}_{v, \tilde v}$ by
\[
{\bf M}_{v, \tilde v}(\omega_\varphi)=2\vol_v(X) {\bf D}_v(\omega_{\varphi}) 
-\int_X \mathring{h}_\varphi v(\mu_\varphi)\omega_\varphi^{[n]}+\int_X \mathring{h}_0 v(\mu_0)\omega_0^{[n]}.
\]
\end{lemma}
\begin{proof}
Using the relations
\[
\begin{split}
dd^{c}\log\left(\frac{v(\mu_\varphi)\omega^{m}_\varphi}{\omega_0^{m}}\right) &= dd^{c}\log v(\mu_\varphi)+2(\Ric(\omega_0)-\Ric(\omega_\varphi)) \\
\Ric(\omega_{0})-\Ric(\omega_\varphi) &= -dd^c\left(\varphi+\frac{1}{2}(h_\varphi-h_0)\right),
\end{split}
\]
we obtain
\begin{equation}\label{log-h}
\log\left(\frac{v(\mu_\varphi)\omega^{[n]}_\varphi}{\omega_0^{[n]}}\right)=\log v(\mu_\varphi) -h_\varphi+h_0-2\varphi+C(\varphi)/{\vol_v(X)},
\end{equation}
where $C(\varphi)$ is a constant depending on $\varphi$ to be determined below. By using the identity (see e.g. \cite[Lemma 2]{lahdili})
\[
\int_X \log v(\mu_\varphi) v(\mu_\varphi)\omega^{[n]}_\varphi = \int_{\Pol_X} \log v(x) v(x) d\mu_{\rm DH} = \int_X \log v(\mu_0) v(\mu_0)\omega^{[n]}_0, 
\]
and integrating both sides of \eqref{log-h} with respect to the measure $v(\mu_\varphi)\omega_\varphi^{[n]}$ gives
\[\begin{split}
 C(\varphi)=&\int_X\log\left(\frac{v(\mu_\varphi)\omega^{[n]}_\varphi}{\omega_0^{[n]}}\right) v(\mu_\varphi)\omega^{[n]}_\varphi +\int_X h_\varphi v(\mu_\varphi)\omega_\varphi^{[n]}-\int_X h_0v(\mu_\varphi)\omega_\varphi^{[n]}\\
&+2\int_X\varphi v(\mu_\varphi)\omega_\varphi^{[n]}-\int_X   \log v(\mu_0) v(\mu_0)\omega^{[n]}_0. 
\end{split}
\]
By Lemma \ref{Mab-Fano} we thus have
\[
C(\varphi)=({\bf M}_{v,\tilde v}+2{\bf I}_{v})(\varphi)+\int_X h_\varphi v(\mu_\varphi)\omega_\varphi^{[n]}-\int_X h_0 v(\mu_0)\omega_0^{[n]}.
\]
 Substituting back in \eqref{log-h} gives
\[
-h_\varphi+\frac{1}{\vol_v(X)}\left(({\bf M}_{v,\tilde v}+2{\bf I}_{v})(\varphi)+\int_X h_\varphi v(\mu_\varphi)\omega_\varphi^{[n]}-\int_X h_0 v(\mu_0)\omega_0^{[n]}\right)=\log\left(\frac{\omega^{[n]}_\varphi}{e^{-2\varphi+h_0}\omega_0^{[n]}}\right). 
\]
Exponentiating and integrating the both sides with respect to the measure $\omega_\varphi^{[n]}$ yields
\[\begin{split}
&\left(\int_Xe^{h_\varphi} \omega_\varphi^{[n]}\right)\exp\left(-\frac{1}{\vol_v(X)}\left(({\bf M}_{v,\tilde v}+2{\bf I}_{v})(\varphi)+\int_X h_\varphi v(\mu_\varphi)\omega_\varphi^{[n]}-\int_X h_0 v(\mu_0)\omega_0^{[n]}\right)\right)\\
&=\int_X e^{-2\varphi+h_0}\omega^{[n]}_0. 
\end{split}
\]
Taking the $\log$ of both sides and comparing with the Definition~\ref{d-v-ding},  we obtain the claim. \end{proof}
The above link between ${\bf M}_{v, \tilde v}$ and ${\bf D}_v$ is the key for the following
\begin{thm}\cite[Thm.3.6]{HL}\label{D/M-coersive} On  a smooth Fano manifold $(X, \T)$ with K\"ahler class $\alpha=2\pi c_1(X)$,  the weighted Mabuchi energy ${\bf M}_{v, \tilde v}$  is coercive with respect to  $\T_{\C}$ iff the weighted Ding functional ${\bf D}_v$ is coercive with respect to  $\T_{\C}$.
\end{thm}
 %
 %

\section{Proof of main results}
\subsection{Mabuchi's observation for $v$-solitons}
Let $\omega\in 2\pi c_1(X)$ be a $\T$-invariant K\"ahler metric. We denote by $\tor_{\omega}$ the space of Killing potentials of vector fields in $\tor$, i.e.
\[ \tor_{\omega}:=\{ f\in C^{\infty}(X) \, | \, \zeta = -\omega^{-1}(df) \in \tor\}, \]
and by $\mathring{\tor}_{\omega}:=\{\mu^\zeta_\omega\mid \zeta\in \tor\}$ be the subspace of \emph{canonically normalized} Killing potentials, i.e.
\[ \mathring{\tor}_{\omega}:= \left\{ f\in \tor_{\omega} \, \Big| \int_X f e^{h_{\omega}}\omega^{[n]}=0\right\},\]
where $h_{\omega}$ is a $\T$-invariant $\omega$-relative Ricci potential, see \eqref{Ricci-potential}.  
By virtue of the normalization \eqref{normalized-mu}, we thus have
\[ \mathring{\tor}_{\omega} = \{ \mu^{\zeta}_\omega \, |\, \zeta \in \tor\}.\]
Taking trace w.r.t. $\omega$ and interior product with $\zeta$ in \eqref{Ricci-potential} gives
\[\Scal(\omega)-2n = -\Delta_{\omega} h_{\omega}, \qquad  \Delta_{\omega} \mu_\omega^{\zeta} - 2\mu_{\omega}^{\zeta} - \langle dh_{\omega}, d\mu_{\omega}^{\zeta}\rangle_{\omega}=0.\]
Taking the $L_2$-product of the first relation with $\mu_{\omega}^{\zeta}$,  integrating by parts  and using the second identity gives
\begin{equation}\label{Fano-Scal}
 \int_X\left(\Scal(\omega)-2(n+1)\right) \mu_{\omega}^{\zeta} \omega^{[n]} = 0. 
 \end{equation}
As we also have $\int_X \mu_{\omega}^{\zeta} e^{h_{\omega}} \omega^{[n]} = 0$ by virtue of \eqref{normalized-mu},  we derive the following
\begin{lemma} 
Suppose $\mathring{h}_{\omega}$ is a $\T$-invariant  $\omega$-relative Ricci potential normalized by 
$\int_X e^{\mathring{h}_{\omega}} \omega^{[n]}= \int_X \omega^{[n]}$. Then the $L_2$-orthogonal projection  
$\proj_{\omega}(e^{{\mathring h}_\omega})$ of $e^{\mathring{h}_\omega}$ to the space $\tor_{\omega}$ of Killing potentials of elements in $\tor$ is given by 
\[ 
\proj_{\omega}(e^{\mathring{h}_\omega}) = (n+1) - \frac{1}{2}\ell_{\rm ext}(\mu_\omega),
\]
where $\ell_{\rm ext}$ is the extremal affine-linear function of $(X, \T, 2\pi c_1(X))$.
\end{lemma}
\begin{proof} 
By the definition \cite{FM} of $\ell_{\rm ext}$,  $\proj_{\omega}(\Scal(\omega))= \ell_{\rm ext}(\mu_{\omega})$. It thus follows from \eqref{Fano-Scal} and \eqref{normalized-mu} that 
\[\proj_{\omega}(e^{h_\omega}) = \lambda \proj_{\omega}\left((n+1) - \frac{1}{2}\Scal(\omega)\right),\] so the claim follows by the normalization of $\mathring{h}_{\omega}$. 
\end{proof}
In particular, if $v$ is affine linear, i.e. if $\omega$ is a Mabuchi soliton, then $v=(n+1)-\frac{1}{2}\ell_{\rm ext}$ is the corresponding positive weight function. 
\subsection{Nakamura's argument for $v$-solitons}

We first extend \cite[Lemma 2.3]{NN24} for an arbitrary weight function $v$.

\begin{lemma}\label{l:Mabuchi}
Let $v$ be a positive smooth weight function on $\Pol_X$ such that 
\[
\vol_v(X):= \int_X{v}(\mu_\omega)\omega^{[n]}= \vol(X).
\]
Then, for any $\T$-invariant K\"ahler metric $\omega \in\mathcal{K}_\alpha^{\T}(X)$ we have
\begin{equation}\label{eq:Mabuchi_Ding}
{\bf M}_{1, 2[(n+1) -v]}(\omega)=\int_X\mathring{h}_{\omega_0} \omega_0^{[n]}-\int_X\mathring{h}_\omega \omega^{[n]}+2\vol(X) {\bf D}_{v}(\omega).
\end{equation}
In particular, 
\[
\frac{1}{2\vol(X)}{\bf M}_{1, 2[(n+1)-v]}(\omega)\geq {\bf D}_{v}(\omega)+C
\]
where $C$ is a constant depending only on $X$ and $v$.
\end{lemma}
 \begin{proof} We shall use the identification $\mathcal{H}^{\T}_{\omega_0}(X)/\R \cong \mathcal{K}_\alpha^{\T}(X)$ and for any $\varphi \in \mathcal{H}^{\T}_{\omega_0}(X)$ we use the under-script $\varphi$ to denote the geometric quantities with respect to the K\"ahler metric $\omega_{\varphi}$.
In the computations below, we also let
\[w:= 2[(n+1)-v].\]
Using the very definition of the $(1,w)$-Mabuchi energy  in \cite{lahdili}, we have 
\[
\begin{split}
(d_{\omega_\varphi }{\bf M}_{1, w})(\dot{\varphi})=&-\int_X\left(\Scal(\omega_{\varphi})-w(\mu_{\omega_\varphi})\right)\dot{\varphi}\omega_{\varphi}^{[n]}\\
=&-\int_X\left(-\Delta_{\omega_\varphi}h_{\omega_\varphi}+2n-w(\mu_{\omega_\varphi})\right)\dot{\varphi}\omega_{\varphi}^{[n]}\\
=&\int_X\left(h_{\omega_\varphi}\Delta_{\omega_\varphi}\dot{\varphi}+(w(\mu_{\omega_\varphi})-2n)\dot{\varphi}\right)\omega_{\varphi}^{[n]}\\
=&-d_{\omega_\varphi}\left(\int_X h_{\omega_\varphi}\omega_{\varphi}^{[n]}\right)(\dot{\varphi}) + \int_X\dot{h}_{\omega_\varphi}\omega_{\varphi}^{[n]}+ \int_X(w(\mu_{\omega_\varphi})-2n)\dot{\varphi}\omega_{\varphi}^{[n]}
\end{split}
\]
where we used that $\Scal(\omega_\varphi)-2n=-\Delta_{\omega_\varphi}h_{\omega_\varphi}$.  We further assume that the Ricci potential is normalized, i.e. $h_{\omega_\varphi}= \mathring{h}_{\omega_\varphi}$, see \eqref{normalized-Ricci}. We then have
\[
\mathring{h}_{\omega_\varphi}=\mathring{h}_{\omega_0}-2\varphi-\log\left(\frac{\omega_\varphi^{m}}{\omega_0^{m}}\right)-\log\int_X e^{h_{\omega_0}-2\varphi}\frac{\omega_0^{[n]}}{\vol(X)}.
\]
The variation with respect to $\varphi$ of both sides of the above equation gives
\[
(d_{\omega_\varphi} \mathring{h}_{\omega_\varphi})(\dot \varphi)=\left(\Delta_{\omega_\varphi}-2\right)\Big(\dot{\varphi} - \int_X\dot{\varphi}e^{\mathring{h}_{\omega_\varphi}}\frac{\omega_\varphi^{[n]}}{\vol(X)}\Big). 
\]
It follows that for $h_{\omega_\varphi}=\mathring{h}_{\omega_\varphi}$
\[
\int_X\dot{h}_{\omega_\varphi}\omega_{\varphi}^{[n]}=2\int_X\dot{\varphi}\left(e^{\mathring{h}_{\omega_\varphi}}-1\right)\omega_\varphi^{[n]}.
\]
Substituting back gives
\[
(d_{\omega_\varphi }{\bf M}_{1, w})(\dot{\varphi})=-d_{\omega_\varphi}\left(\int_X\mathring{h}_{\omega_\varphi}\omega_{\varphi}^{[n]}\right)(\dot{\varphi}) +\int_X\dot{\varphi}\left(2e^{\mathring{h}_{\omega_\varphi}}+ w(\mu_{\omega_\varphi})-2(n+1)\right)\omega_\varphi^{[n]}. 
\]
As the derivative of the $v$-Ding functional is given by \eqref{d-Dv}, and using that  $v:=\frac{1}{2}\left(2(n+1)-w\right)$ satisfies  $\vol_{v}(X)=\vol(X)$,  we get
\[
(d_{\omega_\varphi} {\bf M}_{1, w})(\dot{\varphi})=-d_{\omega_\varphi}\left(\int_X\mathring{h}_{\omega_\varphi}\omega_{\varphi}^{[n]}\right)(\dot{\varphi}) +2\vol(X)(d_{\omega_\varphi}{\bf D}_{v}) (\dot\varphi). 
\]
Using that ${\bf M}_{1, w}(\omega_0)=0={\bf D}_{v}(\omega_0)$, we thus obtain  \eqref{eq:Mabuchi_Ding}. Using  $e^x \geq x+1$ we can write
\[
\int_X\mathring{h}_{\omega_\varphi}\omega_{\varphi}^{[n]}+\vol(X) = \int_X(\mathring{h}_{\omega_\varphi}+1)\omega_{\varphi}^{[n]}\leq \int_X e^{\mathring{h}_{\omega_\varphi}}\omega_{\varphi}^{[n]}=\vol(X).
\]
Letting $C:= \frac{1}{2\vol(X)}\int_X\mathring{h}_{\omega_0} \omega_0^{[n]}$, we  obtain the claimed inequality. \end{proof}

The following is a straightforward generalization for an arbitrary weight  function $v$ of an observation originally made by Nakamura for Mabuchi solitons~\cite{Mabuchi} (see also  \cite{NN24} for the more general $\sigma$-soliton case):

\begin{thm}\label{Cor:Sol->CscK}
Let $v$ be a  smooth positive weight function on $\Pol_X$ such that $\vol_v(X)=\vol(X)$.
If $(X,2\pi c_1(X))$ admits a $v$-soliton then it also admits $(1,2[(n+1)-v])$-cscK metric. In particular, the existence of a Mabuchi soliton implies the existence of a Calabi extremal metric $\omega_E$ in $2\pi c_1(X)$ whose scalar curvature satisfies $\Scal(\omega_E)<2(n+1)$.
\end{thm}
\begin{proof} We let again $w:= 2[(n+1)-v])$. If $(X, 2\pi c_1(X))$ admits a $v$-soliton, then by \cite{HL} ${\bf D}_v$ is coersive with respect to $\T_{\C}$, and so it is ${\bf M}_{1,w}$ by Lemma~\ref{l:Mabuchi},  where we used that both ${\bf D}_v$ and ${\bf M}_{1,w}$ are $\T_{\C}$-invariant (see Lemma~\ref{TC-invariance} and \eqref{eq:Mabuchi_Ding}). Now, we can conclude by using He's extension~\cite{He} of Chen--Cheng's result~\cite{CC},  which holds for the $(1, w)$-cscK problem as $\Fut_{1, w}=0$ and the term $w(\mu_{\omega_{\varphi}})$ stays $C^0(X)$-bounded as $w$ is a fixed smooth function on $\Pol_X$. \end{proof}

\subsection{The converse}
We provide a converse of Theorem~\ref{Cor:Sol->CscK}, as a consequence the YTD correspondence established in \cite{HL}.
\begin{thm}\label{(1,w)-YTD} Let $(X, \T)$ be a smooth Fano manifold and $v$ a smooth  positive weight function on $\Pol_X$ such that  $\vol_v(X)= \vol(X)$. Then the following conditions are equivalent: 
\begin{enumerate}
\item[(i)] $(X, \T)$ admits a $(1,2[(n+1)-v])$-cscK metric in $2\pi c_1(X)$;
\item[(ii)] $(X, K^{-1}_X, \T)$ is uniformly $(1,2[(n+1)-v])$-K-stable relative to $\T_\C$, in the sense of Lahdili~\cite{lahdili}.
\item[(iii)] $(X, K^{-1}_X, \T)$ is uniformly $v$-Ding stable relative to $\T_\C$, in the sense of Han--Li~\cite{HL}.
\item[(iv)] $(X, K^{-1}_X, \T)$ is uniformly $v$-Ding stable relative to $\T_\C$ on special $\T_\C$-equivariant test configurations.
\item[(v)] $(X, \T)$  admits a $v$-soliton in $2\pi c_1(X)$.
\end{enumerate}
\end{thm}
\begin{proof} The implication (i) $\Rightarrow$ (ii) is established in \cite[Cor.1]{AJL}. As a matter of fact, in this reference and in \cite{lahdili}  the authors consider the class of \emph{smooth}  $\T_{\C}$-equivariant test configurations with reduced central fibre in order to define the asymptotic slope ${\bf M}^{\rm NA}_{v, w}(\tstX, \tstL)$ of the weighted Mabuchi energy ${\bf M}_{v, w}$. However, in the case  $v=1$,   which is of interest here, we have that  $ {\bf M}^{\rm NA}_{1, w}(\tstX, \tstL)={\bf M}_{1,2n}^{\rm NA}(\tstX, \tstL) + {\bf I}^{\rm NA}_{2n-w}(\tstX, \tstL)$ where the quantities at the RHS are well-defined for arbitrary  $\T_{\C}$-equivariant test configurations:  ${\bf M}_{1, 2n}^{\rm NA}(\tstX, \tstL)$ is computed in \cite{BHJ, DT, Zak}
whereas the computation of ${\bf I}^{\rm NA}_{2n-w}(\tstX, \tstL)$ for general test configuration appears in 
see \cite{HL} (see also \cite{Yao}).  Thus, the proof  of Cor. 1 in \cite{AJL} extends to yield the desired result. 

Below we outline (ii) $\Rightarrow$ (iv),   noting that (iii) $\Leftrightarrow$ (iv)  $\Leftrightarrow$ (v) is established in \cite{HL}. By \eqref{eq:Mabuchi_Ding},  for any $\T_{\C}$-equivariant test configuration $(\tstX, \tstL)$,  
\[{\bf M}^{\rm NA}_{1, 2[(n+1)-v]}(\tstX, \tstL) - 2\vol(X) {\bf D}_v^{\rm NA}(\tstX, \tstL) = {\bf M}^{\rm NA}_{1, 2n}(\tstX, \tstL) - 2\vol(X){\bf D}_1^{\rm NA}(\tstX, \tstL). \]
By \cite[Theorem 1.3]{Berman-Inv} (which computes ${\bf D}_1^{\rm NA}(\tstX, \tstL)$) and \cite[Theorem A]{BHJ} or \cite[Theorem 5.1] {Zak} (which gives ${\bf M}^{\rm NA}_{1, 2n}(\tstX, \tstL)$), if $(\tstX, \tstL)$ is a \emph{special} test configuration, we have the equality 
\[ {\bf M}^{\rm NA}_{1, 2n}(\tstX, \tstL) = 2\vol(X) {\bf D}_1^{\rm NA}(\tstX, \tstL),\]
which yields ${\bf D}_v^{\rm NA}(\tstX, \tstL)= \frac{1}{2\vol(X)} {\bf M}^{\rm NA}_{1, 2(n+1)-v}(\tstX, \tstL)$,  hence the validity of (ii) $\Rightarrow$ (iv).

Finally, Theorem~\ref{Cor:Sol->CscK} gives (v) $\Rightarrow$ (i).
\end{proof}
The above result yields
\begin{cor} Suppose $(X, \T)$ is a smooth Fano manifold for which the extremal Futaki--Mabuchi affine-linear function $\ell_{\rm ext}$ satisfies on $\Pol_X$
\[ \ell_{\rm ext} < 2(n+1).\]
Then $X$ admits an extremal K\"ahler metric in $2\pi c_1(X)$ iff $(X, K^{-1}_X, \T)$ is uniform relatively K-stable with respect to special $\T_{\C}$-equivariant test configurations in the sense of Sz\'ekelyhidi~\cite{Sz}.
\end{cor}
\begin{proof} 
First note the condition $\ell_{\rm ext} < 2(n+1)$ guarantees that the weight function $(n+1)-\frac{1}{2}\ell_{\rm ext}$ is positive on $\Pol_X$. 
Hisamoto~\cite{Hisamoto} showed that $\T_{\C}$-relative Donaldson--Futaki invariant of a  special $\T_\C$-equivariant test configuration $(\tstX, \tstL)$ in the sense of \cite{Sz} equals ${\bf D}_{((n+1)-\frac{1}{2}\ell_{\rm ext})}^{\rm NA}(\tstX, \tstL)$ (and also to ${\bf M}^{\rm NA}_{1, \ell_{\rm ext}}(\tstX, \tstL)$ by the arguments in Theorem~\ref{(1,w)-YTD}). The results then follows from  Theorem~\ref{(1,w)-YTD}. 
\end{proof}

\begin{rem} The normalization condition for $v$ is necessary for having $\Fut_{1,2[(n+1)-v]}\equiv 0$, see \cite{AJL, lahdili}.
\end{rem}

\begin{cor}\label{extremal/M-soliton} Suppose $X$ is a smooth Fano variety. Then $X$ admits a Mabuchi soliton  $\omega_M\in 2\pi c_1(X)$ iff  $X$ admits an extremal K\"ahler metric $\omega_E\in 2\pi c_1(X)$ such that 
\[\Scal(\omega_E) < 2(n+1).\]
\end{cor}


\end{document}